\documentclass[12pt,reqno]{amsart}

\usepackage{palatino,mathpazo}
\usepackage{enumerate}
\usepackage{amsmath}
\usepackage{amsthm, bbm}
\usepackage{amssymb,mathabx}
\usepackage{hyperref}
\hypersetup{ colorlinks=true, linkcolor=blue, citecolor=blue,
  filecolor=blue, urlcolor=blue}
  \usepackage{graphicx}
\usepackage{caption}
\usepackage{subcaption}

\usepackage[comma, sort&compress]{natbib}
\theoremstyle{plain} 
    \newtheorem{theorem}{Theorem}
    \newtheorem{lemma}[theorem]{Lemma}

\theoremstyle{definition} 

\DeclareMathOperator{\R}{\mathbb{R}}

\DeclareMathOperator{\Z}{\mathbb{Z}}

\usepackage{mathtools}

\DeclareMathOperator{\De}{d}

\DeclareMathOperator{\one}{\mathbbm{1}} 

\newcommand{\I}{\mathcal I}
\newcommand{\Co}{\mathcal C}
\newcommand{\E}{\mathsf{E}}
\newcommand{\Ex}[1]{\mathsf{E}\left[#1 \right]}
\newcommand{\var}[1]{\mathsf{Var}\left[#1 \right]}

\newcommand{\prob}{\mathsf{P}}

\DeclareMathOperator{\e}{e}
\renewcommand{\O}[1]{\mathrm{O}\left(#1\right)} 
\renewcommand{\o}[1]{\mathrm{o}\left(#1\right)} 

\newcommand{\f}{\frac}  

\newcommand{\eps}{\epsilon}
\newcommand{\eq}[1]{\begin{equation#1}}
\newcommand{\eeq}[1]{\end{equation#1}}
\newcommand{\eqa}[1]{\begin{eqnarray#1}}
\newcommand{\eeqa}[1]{\end{eqnarray#1}}
\newcommand{\vr}{\varphi}

\begin{document}

\title[Extremal process of the supercritical Gaussian Free Field]{A note on the extremal process of the supercritical Gaussian Free Field}
\author[A. Chiarini]{Alberto Chiarini}
\thanks{The first author's research is supported by RTG 1845.}
\address{Technische Universit\"at Berlin,
MA 766, Strasse des 17. Juni 136, 10623
Berlin, Germany}
\email{chiarini@math.tu-berlin.de}
\author[A. Cipriani]{Alessandra Cipriani}
\address{Weierstrass Institute, Mohrenstrasse 39, 10117 Berlin, Germany}
\email{Alessandra.Cipriani@wias-berlin.de}

\author[R. S. Hazra]{Rajat Subhra Hazra}
\address{Theoretical Statistics and Mathematics Unit, Indian Statistical Institute, 203, B.T. Road, Kolkata, 700108, India}
\email{rajatmaths@gmail.com}
\begin{abstract}
We consider both the infinite-volume discrete Gaussian Free Field (DGFF) and the DGFF with zero boundary conditions outside a finite box in dimension larger or equal to $3$. We show that the associated extremal process converges to a Poisson point process. The result follows from an application of the Stein-Chen method from \citet{AGG}.
\end{abstract}
\maketitle
\section{Introduction} 
In this article we study the behavior of the extremal process of the DGFF in dimension larger or equal to 3. This extends the result presented in \cite{CCH2015} in which the convergence of the rescaled maximum of the infinite-volume DGFF and the $0$-boundary condition field was shown. It was proved there that the field belongs to the maximal domain of attraction of the Gumbel distribution; hence, a natural question that arises is that of describing more precisely its extremal points. In dimension $2$, this was carried out by \citet{BisLou, BisLou2} complementing a result of \cite{BrDiZe} on the convergence of the maximum; namely, the characterization of the limiting point process with a random mean measure yields as by-product an integral representation of the maximum. The extremes of the DGFF in dimension 2 have deep connections with those of Branching Brownian Motion (\cite{ABBS,ABK,ABK2, ABK3}). These works showed that the limiting point process is a randomly shifted decorated Poisson point process, and we refer to \cite{SubZei} for structural details. In $d\ge 3$, one does not get a non-trivial decoration but instead a Poisson point process analogous to the extremal process of independent Gaussian random variables. To be more precise, we let $E:=[0,\, 1 ]^d \times (-\infty, \,+\infty]$ and $V_N:=[0,\,n-1]^d\cap \Z^d$ the hypercube of volume $N=n^d$. Let $(\vr_\alpha)_{\alpha\in \Z^d}$ be the infinite-volume DGFF, that is a centered Gaussian field on the square lattice with covariance $g(\cdot,\,\cdot)$, where $g$ is the Green's function of the simple random walk. We define the following sequence of point processes on $E$:
\begin{equation}\label{eq:def:pp}
\eta_n(\cdot) :=\sum_{\alpha \in V_N} \varepsilon_{\left(\frac{\alpha}{n}, \frac{\vr_\alpha-b_N}{a_N}\right)}(\cdot)
\end{equation}

where $\varepsilon_x(\cdot)$, $x\in E$, is the point measure that gives mass one to a set containing $x$ and zero otherwise, and
\begin{equation}\label{eq:cs}
b_N:=\sqrt{g(0)}\left[\sqrt{2 \log N}-\frac{\log \log N+\log(4\pi)}{2\sqrt{2 \log N}}\right],\quad\quad a_N:=g(0)(b_N)^{-1}.
\end{equation}
Here $g(0)$ denotes the variance of the DGFF.
Our main result is 
\begin{theorem}\label{thm:ppp}
For the  sequence of point processes $\eta_n$ defined in~\eqref{eq:def:pp} we have that
$$\eta_n \overset{d}\rightarrow \eta,$$
as $n\to+\infty$, where $\eta$ is a Poisson random measure on $E$ with intensity measure given by $\De t \otimes\left( \e^{-z} \De z\right)$ where $ \De t \otimes \De z$ is the Lebesgue measure on $E$, and $\overset{d}\rightarrow $ is the convergence in distribution on $\mathcal M_p(E)$\footnote{$\mathcal M_p(E)$ denotes the set of (Radon) point measures on $E$ endowed with the topology of vague convergence. 
}.
\end{theorem}
The proof is based on the application of the two-moment method of \cite{AGG} that allows us to compare the extremal process of the DGFF and a Poisson point process with the same mean measure. To prove that the two processes converge, we will exploit a classical theorem by Kallenberg.

It is natural then to consider also convergence for the DGFF $(\psi_\alpha)_{\alpha\in \Z^d}$ with zero boundary conditions outside $V_N$. For the sequences of point measures
\begin{equation}\label{eq:def:pp_2}
\rho_n(\cdot) :=\sum_{\alpha \in V_N} \varepsilon_{\left(\frac{\alpha}{n}, \frac{\psi_\alpha-b_N}{a_N}\right)}(\cdot)
\end{equation}
we establish the following Theorem:
\begin{theorem}\label{thm:ppp_2}
For the  sequence of point processes $\rho_n$ defined in~\eqref{eq:def:pp_2} we have that
$$\rho_n \overset{d}\rightarrow \eta,$$
as $n\to+\infty$ in $\mathcal M_p(E)$, where $\eta$ is as in Theorem~\ref{thm:ppp}.
\end{theorem}
The convergence is shown by reducing ourselves to check the conditions of Kallenberg's Theorem on the bulk of $V_N$, where we have a good control on the drift of the conditioned field, and then by showing that the process on the whole of $V_N$ and on the bulk are close as $n$ becomes large.

The outline of the paper is as follows. In Section~\ref{sec:DGFF} we will recall the definition of DGFF and the Stein-Chen method, while Section~\ref{sec:inf} and Section~\ref{sec:fin} are devoted to the proofs of Theorems~\ref{thm:ppp} and \ref{thm:ppp_2} respectively.
\section{Preliminaries}\label{sec:DGFF}
\subsection{The DGFF}
Let $d\ge 3$ and  denote with $\|\,\cdot\,\|$ the $\ell_\infty$-norm on $\Z^d$. Let $\psi=(\psi_\alpha)_{\alpha\in \Z^d}$ be a discrete Gaussian Free Field with zero boundary conditions outside $\Lambda\subset \Z^{d\phantom{d}}$. On the space $\Omega:=\R^{\Z^d}$ endowed with its product topology, its law $\widetilde \prob_\Lambda$ can be explicitly written as
$$\widetilde \prob_\Lambda(\De \psi)=\frac1{Z_\Lambda}\exp\left(-\frac1{2d}\sum_{\alpha,\,\beta\in \Z^d:\,\|\alpha-\beta\|=1}\left(\psi_\alpha-\psi_\beta\right)^2\right)\prod_{\alpha\in \Lambda}\De \psi_\alpha \prod_{\alpha\in \Z^d\setminus\Lambda}\varepsilon_0( \psi_\alpha).$$
In other words $\psi_\alpha=0$ $\widetilde \prob_\Lambda$-a.~s. if $\alpha\in \Z^d\setminus \Lambda$, and $(\psi_\alpha)_{\alpha\in \Lambda}$ is a multivariate Gaussian random variable with mean zero and covariance $(g_\Lambda(\alpha,\,\beta))_{\alpha,\,\beta\in \Z^d}$, where $g_\Lambda$ is the Green's function of the discrete Laplacian problem with Dirichlet boundary conditions outside $\Lambda$. For a thorough review on the model the reader can refer for example to \cite{ASS}. It is known \cite[Chapter 13]{Georgii} that the finite-volume measure $\psi$ admits an infinite-volume limit as $\Lambda \uparrow \Z^d$ in the weak topology of probability measures. This field will be denoted as $\vr=(\vr_\alpha)_{\alpha\in \Z^d}$. It is a centered Gaussian field with covariance matrix $g(\alpha,\,\beta)$ for $\alpha,\,\beta\in \Z^d$. With a slight abuse of notation, we write $g(\alpha-\beta)$ for $g(0,\,\alpha-\beta)$ and also $g_\Lambda(\alpha)=g_\Lambda(\alpha,\,\alpha)$. $g$ admits a so-called random walk representation: if $\mathbb P_\alpha$ denotes the law of a simple random walk $S$ started at $\alpha\in \Z^d$, then
$$
g(\alpha,\,\beta)=\mathbb E_\alpha\left[\sum_{n\ge 0}\one_{\left\{S_n=\beta\right\}}\right].
$$
In particular this gives $g(0)<+\infty$ for $d\ge 3$. A comparison of the covariances in the infinite and finite-volume is possible in the {\em bulk} of $V_N$: for $\delta>0$ this is defined as
\eq{}\label{eq:bulk} V_N^\delta:=\left\{ \alpha\in V_N:\,\|\alpha-\beta\|>\delta n,\,\forall\,\beta\in \Z^d\setminus V_N\right\}.\eeq{} In order to compare covariances in the finite and infinite-volume field, we recall the following Lemma, whose proof is presented in \citet[Lemma 7]{CCH2015}).
\begin{lemma}\label{lemma:almost_g}
For any $\delta>0$ and $\alpha,\,\beta\in V_N^\delta$ one has
\begin{equation}
g(\alpha,\beta)-C_d\left(\delta N^{1/d}\right)^{2-d}\le g_{V_N}(\alpha,\beta)\le g(\alpha,\beta).
\end{equation}
In particular we have, $g_{V_N}(\alpha)=g(0)\left(1+\O{N^{(2-d)/d}}\right)$ uniformly for $\alpha\in V_N^\delta$. 
\end{lemma}

\subsection{The Stein-Chen method}\label{subsec:PPP}
%
As main tool of this article we will use (and restate here) a theorem from \cite{AGG}. Consider a sequence of Bernoulli random variables $(X_\alpha)_{\alpha\in \mathcal I}$ where $ X_\alpha\sim Be(p_\alpha)$ and $\I$ is some index set. For each $\alpha$ we define a subset $B_\alpha\subseteq \mathcal I$ which we consider a ``neighborhood'' of dependence for the variable $X_\alpha$, such that $X_\alpha$ is nearly independent from $X_\beta$ if $\beta\in \I\setminus B_\alpha$. Set
$$
b_1:=\sum_{\alpha\in \I}\sum_{\beta\in B_\alpha}p_\alpha p_\beta,
$$
$$
b_2:=\sum_{\alpha\in \I}\sum_{\alpha\neq \beta\in B_\alpha}\Ex{X_\alpha X_\beta},
$$
$$
b_3:=\sum_{\alpha\in \I}\Ex{\left|\Ex{X_\alpha-p_\alpha\left|\right.\mathcal H_1}\right|}
$$ 
where
$$
\mathcal H_1:=\sigma\left(X_\beta:\,\beta\in \I\setminus B_\alpha\right).
$$
\begin{theorem}[{\citet[Theorem 2]{AGG}}]\label{thm:AGG2}
Let $\mathcal I$ be an index set. Partition the index set $\mathcal I$ into disjoint non-empty sets $\I_1,\,\ldots, \,\I_k$. For any $\alpha\in \I$, let $(X_\alpha)_{\alpha\in\I}$ be a dependent Bernoulli process with parameter $p_\alpha$. Let $(Y_\alpha)_{\alpha\in \I}$ be independent Poisson random variables with intensity $p_\alpha$. Also let
$$W_j:= \sum_{\alpha\in \I_j} X_\alpha \quad\text{ and }\quad Z_j:=\sum_{\alpha\in \I_j} Y_\alpha\quad \text{ and}\quad\lambda_j:= \E[W_j]=\E[Z_j].$$
Then 
\begin{equation}\label{eq:errorbd}
\| \mathcal L( W_1, \ldots, W_k)- \mathcal L( Z_1,\ldots, Z_k)\|_{TV}\le 2\min\left\{1,\, 1.4 \left(\min \lambda_j\right)^{-1/2}\right\}(2b_1+2b_2+b_3)
\end{equation}
where $\|\cdot\|_{TV}$ denotes the total variation distance and $\mathcal L( W_1, \ldots, W_k) $ denotes the joint law of these random variables. 
\end{theorem}
\section{Proof of Theorem~\ref{thm:ppp}: the infinite-volume case}\label{sec:inf}
\begin{proof}
We recall that $E=[0\, ,1 ]^d \times (-\infty, +\infty]$ and $V_N=[0,\,n-1]^d\cap \Z^d$. 
To show the convergence of $\eta_n$ to $\eta$, we will exploit Kallenberg's theorem \cite[Theorem 4.7]{KallenbergRan}. \label{cond}According to it, we need to verify the following conditions: \begin{enumerate}
\item[i)] for any $A$, a bounded rectangle\footnote{A {\em bounded rectangle} has the form $J_1\times\cdots\times J_d$ with $J_i=[0,\,1]\cap (a_i,\,b_i]$, $a_i,\,b_i\in \R$ for all $1\le i\le d$.} in $[0,1]^d$, and $R=(x,y]\subset (-\infty,+\infty]$ 
$$\E[ \eta_n( A\times (x,y])]\to \E[ \eta( A\times (x,y])]=|A|( \e^{-x}-\e^{-y}).$$
We adopt the convention $\e^{-\infty}=0$ and the notation $|A|$ for the Lebesgue measure of $A$.
\item[ii)] For all $k\ge 1$, and $A_1,\, A_2,\,\ldots,\, A_k$ disjoint rectangles in $[0,1]^d$ and $R_1,\,R_2,\,\ldots, \,R_k$, each of which is a finite union o disjoint f intervals of the type $(x,\,y] \subset (-\infty,+\infty]$, 
\begin{align}
&\prob\left( \eta_n( A_1\times R_1)=0,\, \ldots,\, \eta_n(A_k\times R_k)=0\right)\nonumber\\
&\qquad \to \prob\left(\eta( A_1\times R_1)=0, \,\ldots,\, \eta(A_k\times R_k)=0\right)=\exp\left(-\sum_{j=1}^k |A_j| \omega\left(R_j\right)\right)\label{eq:cond_two}
\end{align}
where $\omega(\De z):=\e^{-z}\De z$.
\end{enumerate}
Let us denote by $u_N(z):= a_Nz+b_N$. The first condition follows by Mills ratio
\eqa{}\label{eq:Mills}
\left( 1-\frac1{t^2}\right)\frac{\e^{-{t^2/2}}}{\sqrt{2\pi}t}\le \prob\left(\mathcal N(0,\,1)>t \right)\le \frac{\e^{-{t^2/2}}}{\sqrt{2\pi}t},\quad t>0.
\eeqa{}
More precisely
\begin{align}
\E[ \eta_n( A\times (x,y])]&=\sum_{\alpha\in nA\cap V_N} \prob\left( \vr_\alpha\in (u_N(x), u_N(y)]\right)\nonumber\\
&{\le} \sum_{\alpha\in nA\cap V_N} \left( \frac{\e^{-\frac{u_N(x)^2}{2g(0)}}}{\sqrt{2\pi}u_N(x)}-\frac{\e^{-\frac{u_N(y)^2}{2g(0)}}}{\sqrt{2\pi}u_N(y)}\left(1-\frac1{u_N(y)^2}\right)\right)\label{eq:star}\\
&\le  |n A\cap V_N|\left(\frac{\e^{-x+\o{1}}}{N}-\frac{\e^{-y+\o{1}}}{N}\left(1-\frac1{2 g(0)\log N(1+\o{1})}\right)\right)\nonumber\\
&\to|A|(\e^{-x}-\e^{-y})\label{eq:brueb}.
\end{align}
Similarly, one can plug in \eqref{eq:star} the reverse bounds of \eqref{eq:Mills} to prove the lower bound, and thus condition i).

To show ii), we need a few more details. Let $k\ge 1$, $A_1,\,\ldots, \,A_k$ and $R_1,\ldots, R_k$ be as in the assumptions. Let us denote by
$\I_j= n A_j\cap V_N$ and $\I=\I_1\cup\ldots \cup \I_k$. For $\alpha\in \I_j$ define 
$$X_\alpha:= \one_{\left\{\frac{\vr_\alpha-b_N}{a_N}\in R_j\right\}}$$
and $p_\alpha:= \prob\left( {(\vr_\alpha-b_N)}/{a_N}\in R_j\right)$. Choose now a small $\eps>0$ and fix the neighborhood of dependence $B_\alpha:= B\left(\alpha, (\log N)^{2+2\eps}\right)\cap \I$ for $\alpha\in \I$. Let $W_j:= \sum_{\alpha\in \I_j}X_\alpha$ and $Z_j$ be as in Theorem~\ref{thm:AGG2}.

By the simple observation that 
$$\prob\left( \eta_n( A_1\times R_1)=0, \,\ldots, \,\eta_n(A_k\times R_k)=0\right)= \prob\left( W_1=0,\,\ldots, \,W_k=0\right),$$
 to prove the convergence \eqref{eq:cond_two}, we can use Theorem~\ref{thm:AGG2} and show that the error bound on the RHS of \eqref{eq:errorbd} goes to $0$.

First we bound $b_1$ as follows. By definition of $R_1,\, R_2,\,\ldots, \,R_k$, there exists $z\in \R$ such that $R_j\subset (z,+\infty]$ for $1\le j\le k$. Hence for any $1\le j\le k$,  for any $\alpha\in \I_j$ we have that 
$$p_\alpha=\prob\left( \frac{\vr_\alpha-b_N}{a_N}\in R_j\right)\le \prob( \vr_\alpha> u_N(z))\stackrel{\eqref{eq:Mills}}{\le} \frac{\e^{-\frac{u_N(z)^2}{2g(0)}}}{\sqrt{2\pi}u_N(z)}\sqrt{g(0)}.$$
The bound is independent of $\alpha$ and $j$, therefore for some $C>0$
\eq{}\label{eq:try}b_1\le C N( \log N)^{d(2+2\eps)}\e^{-2z}N^{-2}\to 0.\eeq{}
For $b_2$ note that it was shown in \cite{CCH2015} that for $z\in \R$ and $\alpha\neq \beta\in V_N$
\begin{equation}\label{eq:joint}
\prob(\vr_\alpha>u_N(z), \,\vr_\beta>u_N(z))\le\frac{(2-\kappa)^{3/2}}{\kappa^{1/2}} N^{-{2}/{(2-\kappa)}} \max\left\{ \e^{-2z}\one_{\left\{z\le 0\right\}}, \e^{-2z/(2-\kappa)}\one_{\left\{z>0\right\}}\right\}.
\end{equation}
Here we have introduced
$\kappa:= \mathbb P_0\left( \widetilde H_0=+\infty\right)\in (0,1)$ and $\widetilde H_0=\inf\left\{n\ge 1:\,S_n=0 \right\}$. Observe that for any $1\le j\le k$,  $\alpha\in \I$ and $\beta\in B_\alpha$ one has
$$\E[X_\alpha X_\beta]\le \prob(\vr_\alpha>u_N(z), \vr_\beta>u_N(z))$$
so that by \eqref{eq:joint} we can find some constant $C'>0$ such that
$$b_2\le C' N^{-{\kappa}/(2-\kappa)} (\log N)^{d(2+2\eps)} \max\left\{ \e^{-2z}\one_{\left\{z\le 0\right\}}, \e^{-2z/(2-\kappa)}\one_{\left\{z>0\right\}}\right\}\to 0.$$
Finally we need to handle $b_3$. From Section~\ref{subsec:PPP} we set for $\alpha\in \I$, $\mathcal H_1:=\sigma\left(X_\beta: \beta\in \I\setminus B_\alpha\right)$ and we define
$\mathcal H_2:=\sigma\left( \vr_\beta: \beta\in \I\setminus B_\alpha\right)$. We observe that
\begin{align*}
&b_3= \sum_{\alpha\in \I}\Ex{\left|\Ex{X_\alpha-p_\alpha\left|\right.\mathcal H_1}\right|}\le  \sum_{\alpha\in \I}\Ex{\left|\Ex{X_\alpha\left|\right.\mathcal H_2}-p_\alpha\right|}
\end{align*}
since $\mathcal H_1\subseteq \mathcal H_2$ and using the tower property of the conditional expectation.
Now denote by $U_\alpha:= \Z^d\setminus\left(\I\setminus B_\alpha\right)$.  Let us abbreviate $u_N(R_j):= \{u_N(y):\, y\in R_j\}$. Then for $\alpha\in I_j$ and $1\le j\le k$, by the Markov property of the DGFF \cite[Lemma 1.2]{PFASS} we have that
$$\Ex{X_\alpha\left|\right.\mathcal H_2}=\widetilde \prob_{U_\alpha}(\psi_\alpha+\mu_\alpha\in u_N(R_j))\qquad \prob-a.~s.
$$
where $(\psi_\alpha)_{\alpha\in \Z^d}$ is a Gaussian Free Field with zero boundary conditions outside $U_\alpha$ and 
$$\mu_\alpha=\sum_{\beta\in \I\setminus B_\alpha}\mathbb P_\alpha\left(H_{\I\setminus B_\alpha}<+\infty,\,S_{H_{\I\setminus B_\alpha}}=\beta\right)\varphi_\beta.$$
Here $H_{\Lambda}:=\inf\left\{n\ge 0:\,S_n\in \Lambda\right\}$, $\Lambda\subset\Z^d$.
Now as in \cite{CCH2015} one can show, using the Markov property, that
$$\var{\mu_\alpha}\le \sup_{\beta\in \I\setminus B_\alpha}g(\alpha,\beta)\le \frac{c}{(\log N)^{2(1+\eps)(d-2)}}$$ 
for some $c>0$. Hence we get that there exists a constant $c'>0$ (independent of $\alpha$ and $j$) such that
\eq{}\label{eq:ppp: rate_zero}
\prob\left(|\mu_\alpha|>\left(u_N(z)\right)^{-1-\eps} \right)\le c' \exp\left(-(\log N)^{(2d-5)(1+\eps)} \right).\eeq{}
Recalling that $R_j \subset (z,\,+\infty]$ for all $1\le j\le k$, this immediately shows that for $d\ge 3$
$$\sum_{j=1}^k\sum_{\alpha\in \I_j}\Ex{\left|\widetilde \prob_{U_\alpha}(\psi_\alpha+\mu_\alpha\in u_N(R_j))-p_\alpha\right|\one_{\left\{|\mu_\alpha|>\left(u_N(z)\right)^{-1-\eps} \right\}}}\to 0.$$
So to show that $b_3\to 0$ we are left with proving
\begin{equation}\label{eq:T_1}
\sum_{j=1}^k\sum_{\alpha\in \I_j}\Ex{\left|\widetilde \prob_{U_\alpha}(\psi_\alpha+\mu_\alpha\in u_N(R_j))-p_\alpha\right|\one_{\left\{|\mu_\alpha|\le\left(u_N(z)\right)^{-1-\eps} \right\}}}\to 0.
\end{equation}
We now focus on the term inside the summation. For this, first we write $R_j= \bigcup_{l=1}^m (w_l, r_l]$ with $-\infty<w_1<r_1<w_2<\cdots < r_m\le +\infty$ for some $m\ge 1$. Hence, we can expand the difference in the absolute value of \eqref{eq:T_1} as follows:
\begin{align}
&\left(p_\alpha-\widetilde \prob_{U_\alpha}(\psi_\alpha+\mu_\alpha\in u_N(R_j))\right)\nonumber\\
&=\sum_{l=1}^m\left( \prob( \vr_\alpha\in (u_N(w_l), u_N(r_l)])-\widetilde\prob_{U_\alpha}\left(\psi_\alpha+\mu_\alpha\in (u_N(w_l), u_N(r_l)]\right)\right)\nonumber\\
&=\sum_{l=1}^m \left( \prob( \vr_\alpha> u_N(w_l))-\widetilde\prob_{U_\alpha}(\psi_\alpha+\mu_\alpha> u_N(w_l))\right)\nonumber\\
&-\sum_{l=1}^m \left( \prob( \vr_\alpha> u_N(r_l))-\widetilde\prob_{U_\alpha}\left(\psi_\alpha+\mu_\alpha> u_N(r_l)\right)\right)\label{eq:schop}
\end{align}
(if $r_l=+\infty$ for some $l$, we conventionally set $\prob( \vr_\alpha> u_N(r_l))=0$ and similarly for the other summand).
Using the triangular inequality in \eqref{eq:T_1}, it turns out that to finish it is enough to show that for an arbitrary $w\in \R$,
\begin{equation}
\sum_{\alpha\in \I}\Ex{\left|\widetilde \prob_{U_\alpha}(\psi_\alpha+\mu_\alpha> u_N(w))-\prob( \vr_\alpha> u_N(w))\right|\one_{\left\{|\mu_\alpha|\le\left(u_N(z)\right)^{-1-\eps} \right\}}}\to 0.
\end{equation}
For this, first we show that on $\mathcal Q:=\left\{\prob( \vr_\alpha> u_N(w))>\widetilde \prob_{U_\alpha}(\psi_\alpha+\mu_\alpha> u_N(w))\right\}$
\begin{equation}
\sum_{\alpha\in \I}\Ex{\left(\prob( \vr_\alpha> u_N(w))-\widetilde \prob_{U_\alpha}(\psi_\alpha+\mu_\alpha> u_N(w))\right)\one_{\left\{|\mu_\alpha|\le\left(u_N(z)\right)^{-1-\eps} \right\}}\one_{\mathcal Q}}\to 0.\label{eq:to_prove}
\end{equation}
This follows from the same estimates of $T_{1,2}$ and Claim 6 of \cite{CCH2015}. Indeed on $\mathcal Q\cap\left\{|\mu_\alpha|\le\left(u_N(z)\right)^{-1-\eps} \right\}$
\begin{align*}
&\sum_{\alpha\in \I}\left( \prob( \vr_\alpha> u_N(w))-\prob_{U_\alpha}\left(\psi_\alpha+\mu_\alpha> u_N(w)\right)\right)\\
&\le \sum_{\alpha\in \I}\frac{\sqrt{g(0)}\e^{-\frac{u_N(w)^2}
{2g(0)}}}{\sqrt{2\pi }u_N(w)} \left(1-(1+\o{1})\left(\frac{\sqrt{g_{U_\alpha}
(\alpha)}u_N(w)\e^{\left(1-\frac{g(0)}{g_{U_\alpha}
(\alpha)}\right)\frac{u_N(w)^2}{2g(0)}+\o{1}}}{\sqrt{g(0) }u_N(w)
(1+\o{1})}\right)\right)\\
&\le C N\frac{\sqrt{g(0)}\e^{-\frac{u_N(w)^2}{2g(0)}}}{\sqrt{2\pi }u_N(w)}\o{1}=\o{1}.
\end{align*}

Similarly one can show that on the complementary event $\mathcal Q^{\mathrm c}$ (recall \eqref{eq:to_prove} for the definition of $\mathcal Q$)
$$\sum_{\alpha\in \I}\Ex{\left(\widetilde \prob_{U_\alpha}(\psi_\alpha+\mu_\alpha> u_N(w))-\prob( \vr_\alpha> u_N(w))\right)\one_{\left\{|\mu_\alpha|\le\left(u_N(z)\right)^{-1-\eps} \right\}}\one_{\mathcal Q^{\mathrm c}}}=\o{1}.$$
This shows that $b_3\to 0$. Hence from Theorem~\ref{thm:AGG2} it follows that
$$\left|\prob( W_1=0,\ldots, W_k=0)- \prod_{j=1}^k \prob\left(Z_j=0\right)\right|=\o{1},$$
having used the independence of the $Z_j$'s.
Notice that by definition $Z_j$ is a Poisson random variable with intensity $\sum_{\alpha\in \I_j} \prob\left( (\vr_\alpha-b_N)/a_N\in R_j\right)$. Decomposing $R_j$ as a union of finite intervals and using Mills ratio, similarly to the argument leading to \eqref{eq:brueb}, one has
$$P(Z_j=0)\to \exp( -|A_j| \omega(R_j))$$
(recall $\omega(R_j)=\int_{R_j} \e^{-z} \De z$). Hence it follows that
\eq{}\label{eq:dig}\prod_{j=1}^k \prob(Z_j=0)\to \exp\left(-\sum_{j=1}^k |A_j|\omega(R_j)\right),\eeq{}
which completes the proof of ii) and therefore of Theorem~\ref{thm:ppp}.
\end{proof}
\section{Proof of Theorem~\ref{thm:ppp_2}: the finite-volume case}\label{sec:fin} We will now show the theorem for the field with zero boundary conditions. As remarked in the Introduction, since on the bulk defined in \eqref{eq:bulk} we have a good control on the conditioned field, we will first prove convergence therein, and then we will use a converging-together theorem to achieve the final limit.
We will first need some notation used throughout the Section: first, we consider $(\psi_\alpha)_{\alpha\in V_N}$ with law $\widetilde \prob_N:=\widetilde \prob_{V_N}$. We also use the shortcut $g_N(\cdot,\,\cdot)=g_{V_N}(\cdot,\,\cdot)$. We will need the notation $\Co_K^+(E)$ for the set of positive, continuous and compactly supported functions on $E=[0,1]^d\times (-\infty,+\infty]$. 
\begin{figure}[!ht]
\vspace{-.5in}
\centering
\includegraphics[width=.7\textwidth]{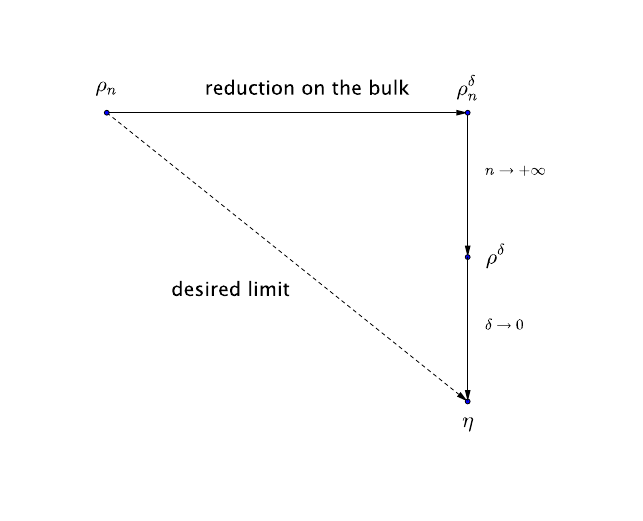}
\vspace{-.5in}
\caption{Sketch of the proof of Theorem~\ref{thm:ppp_2}}
\end{figure}

We first begin with a lemma on the point process convergence on bulk.
Define a point process on $E$ by 
\begin{equation}
\rho_n^\delta(\cdot) =\sum_{\alpha \in V_N^\delta} \varepsilon_{\left(\frac{\alpha}{n}, \frac{\psi_\alpha-b_N}{a_N}\right)}(\cdot).
\end{equation}
\begin{lemma}\label{lemma:bulk}
Let $\delta>0$. On $\mathcal M_p(E)$, $\rho_n^{\delta}\overset{d}\to \rho^{\delta}$ where $\rho^\delta$ is a Poisson random measure with intensity $\De t_{|_{[\delta,1-\delta]^d}}\otimes \left(\e^{-x}\De x\right)$\footnote{$\De t_{|_{[\delta,1-\delta]^d}}$ is the restriction of the Lebesgue measure to $[\delta,1-\delta]^d$.}.
\end{lemma}
\begin{proof}
We will show i) and ii) of Page~\pageref{cond} (and from which we will borrow the notation starting from now).
 \\
i) We begin with an upper bound on $\widetilde\E_N\left[{\rho_n^\delta(A\times (x,y])}\right]$:
\begin{align}
\sum_{\alpha\in nA\cap V_N^\delta}&\widetilde\prob_N(\psi_\alpha>u_N(x))-\widetilde\prob_N(\psi_\alpha>u_N(y))\nonumber\\
&\stackrel{\eqref{eq:Mills}}{\le}\sum_{\alpha\in nA\cap V_N^\delta} \frac{\e^{-\frac{u_N(x)^2}{2g_N(\alpha)}}}{\sqrt{2\pi}u_N(x)}\sqrt{g_N(\alpha)}-\frac{\e^{-\frac{u_N(y)^2}{2g_N(\alpha)}}}{\sqrt{2\pi}u_N(y)}\sqrt{g_N(\alpha)}\left(1+\o{1} \right)\nonumber\\
&\stackrel{\text{Lemma\,\ref{lemma:almost_g}}}{=}\sum_{\alpha\in nA\cap V_N^\delta} \frac{\e^{-\frac{u_N(x)^2}{2g(0)(1+c_n)}}}{\sqrt{2\pi}u_N(x)}\sqrt{g(0)}(1+c_n)-\frac{\e^{-\frac{u_N(y)^2}{2g(0)(1+c_n)}}}{\sqrt{2\pi}u_N(y)}\sqrt{g(0)}\left(1+c_n\right)\nonumber\\
&\stackrel{n\to+\infty}{\longrightarrow}(\e^{-x}-\e^{-y})\left|A\cap [\delta,\,1-\delta]^d\right|.\label{eq:ger}
\end{align}
We stress that in the second step the error term $c_n:=\O{n^{2-d}}$ coming from Lemma~\ref{lemma:almost_g} guarantees the convergence in the last line. The lower bound follows similarly.\\
ii)  To show the second condition we again use Theorem~\ref{thm:AGG2}. Let $A_1,\ldots, A_k$ and $R_1,\ldots, R_k$ be as in proof of Theorem~\ref{thm:ppp}.  Let $\I_j:=n A_j\cap V_N^\delta$ and $\I=\I_1\cup\cdots \cup \I_k$. For $\eps>0$ we are setting $B_\alpha:=B\left(\alpha,\,(\log N)^{2(1+\eps)}\right)\cap \I$.  Note that, albeit slightly different, we are using the same notations for the neighborhood of dependence and the index sets of Section~\ref{sec:inf}, but no confusion should arise. Observe that there exists $z\in \R$ such that for all $1\le j\le k$, $R_j\subset (z,\infty]$; we have
$$
p_\alpha=\widetilde\prob_N\left( \f{\psi_\alpha-b_N}{a_N}\in u_N(R_j)\right)\le\widetilde\prob_N\left({\psi_\alpha}>u_N(z)\right)\stackrel{\eqref{eq:Mills}}{\le} \frac{\e^{-\frac{u_N(z)^2}{2g(0)}}}{\sqrt{2\pi}u_N(z)}\sqrt{g(0)}
$$
where we have also used the fact that $g_N(\alpha)\le g(0)$. The bound on $b_1$ (cf. Theorem~\ref{thm:AGG2}) follows exactly as in \eqref{eq:try} and yields that, for some $C>0$,
$$
b_1\le C N  ( \log N)^{d(2+2\eps)}\e^{-2z}N^{-2}\to 0.
$$
The calculation of $b_2$ can be performed similarly using the covariance matrix of the vector $(\psi_\alpha,\,\psi_\beta)$, $\alpha\neq\beta\in V_N^\delta$ and Lemma~\ref{lemma:almost_g}. This gives that for some $C,\,C'>0$ independent of $\alpha,\beta\in V_N^\delta$
\begin{align*}
& b_2\le\sum_{\alpha\in \I}\sum_{\beta\in B_\alpha}\frac{C}{\log N}\exp\left(-  \f{u_N(z)^2}{g(0)+g(\alpha-\beta)}\left(1+\O{N^{(2-d)/d}}\right)\right)\\
&\le C' N^{-\kappa/(2-\kappa)}(\log N)^{2d(1+\eps)} \max\left\{ \e^{-2z}\one_{\left\{z\le 0\right\}}, \,\e^{-2z/(2-\kappa)}\one_{\left\{z>0\right\}}\right\}\to 0
\end{align*}
(cf. \citet{CCH2015}). We will now pass to $b_3$. We repeat our choice of $\mathcal H_1=\sigma\left(X_\beta:\,\beta\in \I\setminus B_{\alpha}\right)$ and $\mathcal H_2=\sigma\left(\psi_\beta:\,\beta\in \I\setminus B_{\alpha}\right)$ so that $b_3$ becomes
$$
\sum_{j=1}^k\sum_{\alpha\in \I_j}\widetilde{\mathsf E}_N\left[\left|\widetilde{\mathsf E}_N\left[X_\alpha-p_\alpha|\mathcal H_1\right]\right|\right]\le\sum_{j=1}^k\sum_{\alpha\in \I_j}
\widetilde{\mathsf E}_N\left[\left|\widetilde{\mathsf E}_N\left[X_\alpha|\mathcal H_2\right]-p_\alpha\right|\right].
$$
We define $U_\alpha:=V_N\setminus(\I\setminus B_\alpha)$. By the Markov property of the DGFF
\eq{}\label{eq:maya}
\widetilde{\mathsf E}_N\left[{X_\alpha\left|\right.\mathcal H_2}\right]=\widetilde \prob_{U_\alpha}(\xi_\alpha+h_\alpha\in u_N(R_j))\quad \widetilde \prob_N-a.~s.
\eeq{}
for $(\xi_\alpha)_{\alpha\in \Z^d}$ a DGFF with law $\widetilde \prob_{U_\alpha}$ and $(h_\alpha)_{\alpha\in \Z^d}$ is independent of $\xi$. From \citet{CCH2015} we can see that, for $\alpha\in V_N^\delta$ and $N$ large enough such that $B\left(\alpha,\,(\log N)^{2(1+\eps)}\right)\subsetneq V_N$,
\eqa{*}
\var{h_\alpha}&=&\sum_{\beta\in \I\setminus B_\alpha}\mathbb P_\alpha\left(H_{\I\setminus B_\alpha}<+\infty,\,S_{H_{\I\setminus B_\alpha}}=\beta\right)g_N(\alpha,\,\beta)\\
&\le& \sup_{\beta\in\I\setminus B_\alpha}g_{N}(\alpha,\,\beta)\le\frac{c}{(\log N)^{2(1+\eps)(d-2)}}.
\eeqa{*}
This yields
\eq{}\label{eq:pani}
\sum_{j=1}^k\sum_{\alpha\in \I_j}\widetilde{\mathsf E}_N\left[\left|\widetilde \prob_{ U_\alpha}(\xi_\alpha+h_\alpha)>u_{N}(R_j))-p_\alpha\right|\one_{\left\{|h_\alpha|>\left(u_{N}(z)\right)^{-1-\eps} \right\}}\right]\to 0.
\eeq{}
It then suffices to show
\eq{}\label{eq:pani2}
\sum_{j=1}^k\sum_{\alpha\in \I_j}\widetilde{\mathsf E}_N\left[\left|\widetilde \prob_{ U_\alpha}(\xi_\alpha+h_\alpha)>u_{N}(R_j))-p_\alpha\right|\one_{\left\{|h_\alpha|\le\left(u_{N}(z)\right)^{-1-\eps} \right\}}\right]\to 0.
\eeq{}
One sees that the breaking up \eqref{eq:schop} can be performed also here replacing $\vr_\alpha$ and $\psi_\alpha$ (with their laws) with $\psi_\alpha$ and $\xi_\alpha$ (with their laws) respectively, and $\mu_\alpha$ with $h_\alpha$. Accordingly, it is enough to show that 
\begin{align}\label{eq:spar}
&\sum_{\alpha\in \I}\widetilde{\mathsf{E}}_N\left[\left|\widetilde \prob_{U_\alpha}(\xi_\alpha+h_\alpha> u_N(w))-\widetilde\prob_N( \psi_\alpha> u_N(w))\right|\one_{\left\{|h_\alpha|\le\left(u_N(z)\right)^{-1-\eps} \right\}}\right]\to 0
\end{align}
for all $w\in \R$. To this aim, we choose for any $w\in \R$ the event 
$$\mathcal Q':=\left\{\widetilde \prob_N( \psi_\alpha>u_N(w))>\widetilde \prob_{U_\alpha}(\xi_\alpha+h_\alpha> u_N(w))\right\}$$
and we proceed as in \eqref{eq:to_prove} with the help of Lemma~\ref{lemma:almost_g} to show \eqref{eq:spar}. Given this, the convergence $b_3\to 0$ is finally ensured. Hence we can conclude that
$$
\|\mathcal L(W_1,\,\ldots,\,W_k)- \mathcal L(Z_1,\,\ldots,\,Z_k)
\|_{TV}\to 0$$
where $Z_j$ are i.~i.~d. Poisson of mean $p_\alpha$. By Mills ratio, as in \eqref{eq:ger} we see that
$$
\prob(Z_j=0)\to \exp\left(-{\left|A_j\cap[\delta,\,1-\delta]^d\right|}\omega(R_j)\right).
$$
From this it follows that the two conditions i) and ii) of Kallenberg's Theorem are satisfied, and thus we obtain the convergence to a Poisson point process with mean measure given in i).
\end{proof}
\begin{proof}[Proof of Theorem~\ref{thm:ppp_2}]
$\mathcal M_p(E)$ is a Polish space with metric $\mathrm{d}_p$:
$$
\mathrm{d}_p(\mu,\,\mu')=\sum_{i\ge 1}\frac{\min\left\{\left|\mu(f_i)-\mu'(f_i)\right|,\,1 \right\}}{2^i},\quad \mu,\,\mu'\in \mathcal M_p(E)
$$
for a sequence of functions $f_i\in \Co_K^+(E)$ (cf. \citet[Section 3.3]{Resnick}). Therefore we are in the condition to use a converging-together theorem \cite[Theorem 3.5]{ResnickHeavy}, namely to prove that $\rho_n\overset{d}\to \eta$ it is enough to show the following:
\begin{itemize}
\item[(a)] $\rho_n^\delta\overset{d}\to \rho^\delta$, as $n\to+\infty$.
\item[(b)] $\rho^{\delta}\overset{d}\to \eta$ as $\delta\to 0$.
\item[(c)] For every $\eps>0$, 
\begin{equation}\label{eq:kw}
\lim_{\delta\to 0} \lim_{n\to+\infty}\widetilde \prob_N\left(\mathrm{d}_p\left( \rho_n,\rho_n^\delta\right)>\eps\right)=0.
\end{equation}
\end{itemize}
Note that by Lemma~\ref{lemma:bulk}, (a) is satisfied. 
 For $f\in \Co_K^{+}(E)$, the Laplace functional  of $\rho^{\delta}$ is given by (cf. \citet[Prop. 3.6]{Resnick})
$$\Psi_\delta(f):= \Ex{ \exp\left(-\rho^{\delta}(f)\right)}= \exp\left(-\int_E \left(1-\e^{-f(t,x)}\right)\De t_{|_{[\delta,1-\delta]^d}}\e^{-x}\De x\right).$$
Hence by the dominated convergence theorem we can exchange limit and expectation as $\delta\to 0$ to obtain that
$$\Psi_\delta(f)\to  \exp\left(-\int_E \left(1-\e^{-f(t,x)}\right)\De t\e^{-x}\De x\right)$$
and the right hand side is the Laplace functional of $\eta$ at $f$. This shows (b). 

Hence to complete the proof it is enough to show~\eqref{eq:kw}. Thanks to the definition of the metric $\mathrm{d}_p$ it suffices to prove that for $f\in \Co_K^{+}(E)$ and for $\eps>0$
$$\limsup_{\delta\to 0}\lim_{n\to+\infty}\widetilde\prob_N\left( \left|\rho_n(f)-\rho_n^\delta(f)\right|>\eps\right)=0.$$
Without loss of generality assume that the support of $f$ is contained in $[0,1]^d\times [z_0,\,+\infty)$ for some $z_0\in \R$.  Choosing $n$ large enough such that $u_N(z_0)>0$ and $g_N(\alpha)\le g(0)$, we obtain that
\begin{align*}
\widetilde\E_N&\left[\left|\rho_n(f)- \rho_n^{\delta}(f)\right|\right]= \widetilde \E_N\left[\sum_{\alpha\in V_N\setminus V_N^{\delta} } f\left(\frac{\alpha}{n}, \frac{\psi_\alpha-b_N}{a_N}\right)\one_{\left\{\frac{\psi_\alpha-b_N}{a_N}>z_0\right\}}\right]\\
&\le \sup_{z\in E}|f(z)| \sum_{\alpha\in V_N\setminus V_N^{\delta}} \widetilde \prob_N\left(\frac{\psi_\alpha-b_N}{a_N}>z_0\right)\stackrel{\eqref{eq:Mills}}{\le}C\sum_{\alpha\in V_N\setminus V_N^\delta} \f{\e^{-u_N(z_0)^2/g(0)}}{\sqrt{2\pi}u_N(z_0)}\sqrt{g(0)}\\
&\le C'\left(1-(1-2\delta)^d\right)\e^{-z_0} 
\end{align*}
as $n\to+\infty$ for some positive constants $C,\,C'$. Now letting $\delta\to 0$ the result follows and this completes the proof.

\end{proof}
\bibliographystyle{abbrvnat}

\bibliography{literatur}

\end{document}